\documentclass{amsart}
\usepackage{amssymb, latexsym, comment, amsmath, amsthm, upgreek, epsf, epsfig, color, natbib, graphicx}

\makeatletter
\def\ps@pprintTitle{%
\let\@oddhead\@empty
\let\@evenhead\@empty
\def\@oddfoot{\centerline{\thepage}}%
\let\@evenfoot\@oddfoot}
\makeatother

\newtheorem{thm}{Theorem}[section]

\newtheorem{prop}[thm]{Proposition}
\newtheorem{definition}[thm]{Definition}

\begin{document}

\title{Existence and Non-existence of Half-Geodesics on $S^2$}
\author{Ian M. Adelstein} 
\address{Department of Mathematics, Dartmouth College, Hanover, NH 03755 United States}
\email{iadelstein[at]gmail[dot]com}
\begin{abstract} In this paper we study half-geodesics, those closed geodesics that minimize on any subinterval of length $L/2$, where $L$ is the length of the geodesic. For each nonnegative integer $n$, we construct Riemannian manifolds diffeomorphic to $S^2$ admitting exactly $n$ half-geodesics. Additionally, we construct a sequence of Riemannian manifolds, each of which is diffeomorphic to $S^2$ and admits no half-geodesics, yet which converge in the Gromov-Hausdorff sense to a limit space with infinitely many half-geodesics. 
\end{abstract}


\maketitle

%
\section{Introduction}

The existence of closed geodesics on Riemannian manifolds is one of the foundational questions in the field of global differential geometry. The classical result of Cartan, which states that on closed Riemannian manifolds every nontrivial homotopy class of curves contains a closed geodesic, addressed this problem for manifolds having nontrivial fundamental group. The question of existence on simply-connected manifolds would prove to be much more difficult. In 1905, Poincar\'e \cite{Poin} established the existence of a closed geodesic on any surface analytically equivalent to $S^2$. Birkhoff \cite{Birk} proved in 1917 the existence of a closed geodesic on any manifold homeomorphic to $S^n$. On surfaces homeomorphic to $S^2$, Lyusternik and Schnirelmann \cite{LS} in 1929 showed that there exist at least three geometrically distinct closed geodesics, and in 1991 Bangert \cite{Ban} and Franks \cite{Fra} showed that there exist infinitely many. 

The study of the existence of closed geodesics on arbitrary compact manifolds would yield exciting new techniques. The celebrated theorem of Lyusternik and Fet \cite{LF} in 1951 states that there exists a closed geodesic on any compact Riemannian manifold. They employed Morse's calculus of variations in showing that the energy functional $E: \Omega \to \mathbb{R}$ on the loop space of an arbitrary compact manifold has a critical point. Then in 1969 Gromoll and Meyer \cite{Grom} established that there are infinitely many geometrically distinct closed geodesics on compact simply-connected manifolds under a relatively weak topological restriction on the loop space. The question of the existence of infinitely many geometrically distinct closed geodesics on an arbitrary compact manifold remains open. 

A defining property of a geodesic is that it is a locally distance minimizing curve. It is clear that a nontrivial closed geodesic can never be a \emph{globally} distance minimizing curve. Indeed, a closed geodesic cannot minimize past half its length, as traversing the geodesic in the opposite direction always provides a shorter path. It is therefore natural to consider the largest interval on which a given closed geodesic is distance minimizing. This led Sormani \cite{Sor} to consider the notion of a $1/k$-geodesic.

\begin{definition} A $1/k$-geodesic on a compact Riemannian manifold (or more generally on a compact length space) is a closed geodesic $\gamma \colon S^1 \to M$ which is minimizing on all subintervals of length $l(\gamma)/k$, i.e. $$d(\gamma(t),\gamma(t+2\pi/k))=l(\gamma)/k \hspace{4mm} \forall t \in S^1$$ 
\end{definition}

Sormani introduced this notion of a 1/k geodesic to study the behavior of closed geodesics under deformations of Riemannian manifolds. She showed that the 1/k-geodesics persist under the Gromov-Hausdorff convergence of Riemannian manifolds (Theorem~\ref{conv2}). This result should be viewed in contrast to the behavior of arbitrary closed geodesics, which can disappear under Gromov-Hausdorff convergence (cf. \cite[Example 7.2]{Sor}).

The half-geodesics, those closed geodesics that minimize on any subinterval of length $l(\gamma)/2$, are of inherent geometric interest. They are the closed geodesics that have the maximal minimizing property. The existence of a half-geodesic provides an upper bound on the length $L$ of the shortest nontrivial closed geodesic; we have $L \leq  2 \, \text{diam}(M)$ when the set of half-geodesics on a manifold is nonempty. In \cite{Ade2}, the author studied half-geodesics by first providing a relationship between the half-geodesics and the Grove-Shiohama critical points of the distance function. 

We consider the question of the existence of half-geodesics on Riemannian manifolds. Analogous to the existence of closed geodesics on arbitrary compact manifolds, the non-simply-connected case is straightforward. Sormani \cite[Lemma 4.1]{Sor} shows that a closed geodesic which is the shortest among all non-contractible closed geodesics is a half-geodesic; i.e.,~the systole of a non-simply-connected manifold is always a half-geodesic.

Although Bangert and Franks showed that a manifold diffeomorphic to $S^2$ admits infinitely many closed geodesics, it is known that metrics on the $2$-sphere need not admit any half-geodesics. Indeed, using Clairaut's relation, Wing Kai Ho \cite{Ho} produced surfaces of revolution diffeomorphic to the $2$-sphere that do not admit half-geodesics. Additionally, Balacheff, Croke and Katz \cite{Bal} constructed a Zoll surface whose closed geodesics have length greater than twice the diameter, hence fails to admit a half-geodesic. The following result provides new examples of metrics on $S^2$ which do not admit half-geodesics, and shows that the absence of half-geodesics is not preserved under Gromov-Hausdorff convergence.

\begin{thm} \label{none} There exists a sequence of Riemannian manifolds, each of which is diffeomorphic to $S^2$ and admits no half-geodesics, which converge in the Gromov-Hausdorff sense to a limit space that has infinitely many half-geodesics. 
\end{thm}

The following is a positive result on the existence of half-geodesics on Riemannian manifolds diffeomorphic to $S^2$.

\begin{thm} \label{exact} For each nonnegative integer $n$, there exist Riemannian manifolds diffeomorphic to $S^2$ admitting exactly $n$ half-geodesics.
\end{thm}

The paper proceeds as follows.  In Section~\ref{poly} we study the existence and non-existence of half-geodesics on doubled regular polygons (Proposition~\ref{polygon}). In Section~\ref{half} we construct sequences of Riemannian manifolds diffeomorphic to $S^2$ that converge in the Gromov-Hausdorff sense to these doubled regular polygons. We use a result by Cheeger \cite[Theorem 2.1]{Che2} to show that these sequences admit positive uniform lower bounds on the length of closed geodesics (Proposition~\ref{no_short_closed}). We show the existence and non-existence of half-geodesics on these sequences by applying Sormani's result (Theorem~\ref{conv2}) concerning the persistence of 1/k-geodesics under Gromov-Hausdorff convergence. The manifolds in these sequences are then used to prove Theorems~\ref{none} and~\ref{exact}.

%
%
\section{Half-Geodesics on Doubled Polygons}\label{poly}

In this section we study the half-geodesics on doubled regular $n$-gons $X_n$. These spaces are convex compact length spaces, not Riemannian manifolds, and will serve as the limiting spaces for the sequences of Riemannian manifolds that we construct in Section~\ref{half}. 

\begin{prop}\label{polygon} Let $X_n$ be a doubled regular $n$-gon.
\begin{enumerate}
\item If n is odd then $X_n$ has no half-geodesics.
\item If n is even then $X_n$ has exactly n/2 half-geodesics:~those curves which pass through the center of each face and perpendicularly through parallel edges. 
\end{enumerate}
\end{prop}
\begin{proof} 

We first note that a half-geodesic $\gamma \colon S^1 \to X_n$ contains an edge point because closed geodesics on $X_n$ contain points from each face and therefore cross an edge (crossing at a vertex would violate the local length minimization property of the geodesic). We label this edge point $p=\gamma(0)$ and parameterize the curve $\gamma$ by a circle of length $2\pi$ so that $p$ and $q=\gamma(\pi)$ are at distance $l(\gamma)/2$. The two halves of $\gamma$ minimize between $p$ and $q$ which shows that $q$ is a cut point of $p$ and therefore also an edge point. The fact that these two halves of $\gamma$ connect to form a closed geodesic implies that $p$ and $q$ are on parallel (i.e.~opposite) edges and that $\gamma$ crosses these edges perpendicularly. As doubled regular $n$-gons with an odd number of sides do not have parallel edges we conclude that these spaces do not admit half-geodesics.

Let $\gamma \colon S^1 \to X_n$ be a half-geodesic on a doubled regular $n$-gon with an even number of sides. As above we have that $\gamma$ is the concatenation of the two straight line paths connecting a pair of points on parallel edges. We show that $\gamma$ contains the center point of each face. Suppose this is not the case. Then letting $\gamma(0)$ and $\gamma(\pi)$ be the edge points of $\gamma$ we see that $\gamma(\pi/2)$ and $\gamma(3\pi/2)$ are a pair of off-center points on opposite faces which lie exactly halfway between the parallel edges, and are therefore connected via a shorter curve which crosses a third edge of the $n$-gon. We conclude that $\gamma$ does not minimize between this pair of points and is therefore not a half-geodesic.

We therefore identify $n/2$ candidate half-geodesics: those curves which pass through the center of each face and perpendicularly through the midpoints of parallel edges. We show that each of these curves $\gamma \colon S^1 \to X_n$ is a half-geodesic, i.e.~that $\gamma$ minimizes between $p=\gamma(t)$ and $q=\gamma(t+\pi)$ for every $t \in S^1$. First note that if $p$ and $q$ are edge points then $\gamma$ minimizes between them. If the points $p$ and $q$ are on opposite faces then any shorter curve joining them would have to cross an edge of the $n$-gon at a point distinct from the two edge points of $\gamma$. Consider for a moment a single-sided regular $n$-gon where $p$ and $q$ have both been placed as they were on their respective faces (see Figure~\ref{fig1}). We inscribe an ellipse in this $n$-gon with focal points $p$ and $q$ and major axis length $l(\gamma)/2$. The $n$-gon is tangent to the ellipse along its major axis, and we see that all of the other edge points are at a combined distance greater than or equal to $l(\gamma)/2$ from $p$ and $q$. Back on the doubled $n$-gon, we conclude that a shorter curve joining $p$ and $q$ is impossible. We have therefore shown that $\gamma$ is a half-geodesic.  
\end{proof}

\begin{figure}[!htb]\centering
\includegraphics[width=.65\textwidth]{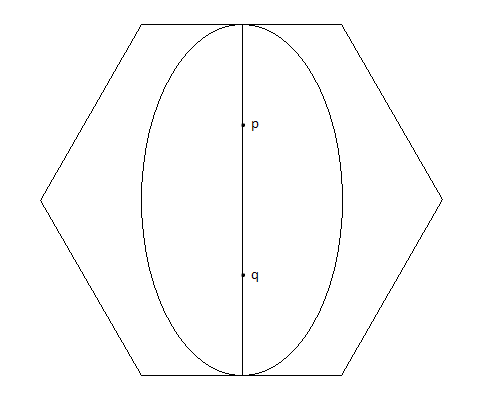}
\caption{A single-sided regular $n$-gon. The ellipse with focal points $p$ and $q$ and major axis length $l(\gamma)/2$ is tangent to the $n$-gon along its major axis. All other edge points of the $n$-gon are at a combined distance greater than $l(\gamma)/2$ from $p$ and $q$.} 
\label{fig1}
\end{figure}

%
%
%
\section{Construction of the Surfaces}\label{half}

The doubled regular $n$-gons studied in Proposition~\ref{polygon} are convex compact length spaces. We study the existence of half-geodesics on Riemannian manifolds by first constructing for each $n \geq 3$ a sequence of piecewise smooth manifolds converging in the Gromov-Hausdorff sense to a doubled regular $n$-gon $X_n$. Start by choosing a sequence $\epsilon_{n,i} \to 0$ with $\epsilon_{n,i} > 0$ and let $T_{\epsilon_{n,i}}(X_n)$ be the $\epsilon_{n,i}$ tubular neighborhood of $X_n$. Then let $Y_{n,i} = \partial T_{\epsilon_{n,i}}(X_n)$ be the boundary of this $\epsilon_{n,i}$ tubular neighborhood. We choose the $\epsilon_{n,i}$ sufficiently small so that the $Y_{n,i}$ consist of two flat faces each isometric to a face of $X_n$, joined by half-cylinders along the edges and $1/n$-sections of spheres between the vertices (see Figure~\ref{fig2}). As submanifolds of $\mathbb{R}^3$ the $Y_{n,i}$ are piecewise smooth surfaces. We note that the $Y_{n,i}$ converge in the Gromov-Hausdorff sense to $X_n$. 

\begin{figure}[!htb]\centering
    \includegraphics[width=.73\textwidth]{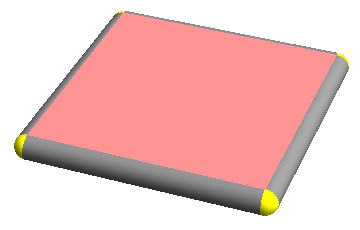}
    \caption{A sample $Y_{4,i}$. We see the pink top face, the grey half-cylinders joining the edges, and the yellow $1/4$-sections of spheres between the vertices.} 
    \label{fig2}
    \end{figure}

Next let $M_{n,i}$ be Riemannian manifolds obtained by smoothing the $Y_{n,i}$ along the boundaries of their $n$ spherical regions. We smooth symmetrically, maintaining two equivalent hemispheres (the top and bottom halves of the surface) separated by an equator curve which passes through the middle of each of the cylindrical and spherical regions. The resultant $M_{n,i}$ are Riemannian manifolds which converge in the Gromov-Hausdorff sense to $X_n$. 

\begin{prop}\label{no_short_closed} For each $n\geq3$ the sequence $\{M_{n,i}\}_{i=1}^{\infty}$ constructed above admits a positive uniform lower bound on the length of closed geodesics.
\end{prop}
\begin{proof} We apply the following result by Cheeger \cite[Theorem 2.1]{Che2}: given constants $d, v >0$ and $H$ there exists a positive constant $c_m(d,v,H)$ such that if $M$ is a Riemannian $m$-manifold with $\text{diam}(M)<d$, $\text{vol}(M)>v$ and sectional curvature $k \geq H$ then every closed geodesic in $M$ has length greater than $c_m(d,v,H)$. For each $n\geq3$ the sequence $\{M_{n,i}\}_{i=1}^{\infty}$ admits a uniform upper bound of $2\, \text{diam}(X_n)$ on diameter and uniform lower bounds of $\text{vol}(X_n)$ on volume and $H=0$ on sectional curvature so that the Cheeger result applies.
\end{proof}

We will use the following convergence result due to Sormani to study the half-geodesics on the sequences $\{M_{n,i}\}_{i=1}^{\infty}$.

\begin{thm}[\cite{Sor}, Theorem 7.1]\label{conv2} Let $M_i$ be a sequence of compact Riemannian manifolds converging in the Gromov-Hausdorff sense to a compact length space $X$. Let $\gamma_i \colon S^1 \to M_i$ be a sequence of $1/k$-geodesics. Then a subsequence of the $\gamma_i$ converge point-wise to a continuous curve $\gamma \colon S^1 \to X$, and $\gamma$ is either a $1/k$-geodesic or trivial. 
\end{thm}

\begin{prop}\label{odd3} For each odd integer $n \geq 3$ only finitely many of the manifolds in the sequence $\{M_{n,i}\}_{i=1}^{\infty}$ constructed above can admit a half-geodesic.
\end{prop}
\begin{proof} This follows immediately from Proposition~\ref{polygon}(1), Proposition~\ref{no_short_closed} and Theorem~\ref{conv2}.
\end{proof}

\begin{proof}[Proof of Theorem~\ref{none}] We choose a sequence $\epsilon_n \to 0$ such that the smoothed Riemannian manifolds $M_n \approx \partial T_{\epsilon_n}(X_{2n+1})$ do not admit any half-geodesics. The sequence $M_n$ converges in the Gromov-Hausdorff sense to a doubled disk, a compact length space admitting infinitely many half-geodesics.
\end{proof}

\begin{prop}\label{even} For each $n \geq 2$ there exists a sequence of Riemannian manifolds diffeomorphic to $S^2$ each of which admits exactly $n$ half-geodesics.
\end{prop}
\begin{proof} 

Let $X_{2n}$ be a doubled regular $2n$-gon. Recall from Proposition~\ref{polygon} that $X_{2n}$ admits exactly $n$ half-geodesics. We will first show that each manifold in the sequence $\{M_{2n,i}\}_{i=1}^{\infty}$ constructed above admits at least $n$ half-geodesics, and then that only finitely many of the $M_{2n,i}$ can admit more than $n$ half-geodesics.

We argue that the $n$ meridians of $M_{2n,i}$, those curves which pass through the center of each face and perpendicularly through parallel cylindrical edges, are half-geodesics. Indeed, because the smoothing of the $M_{2n,i}$ occurs only in a neighborhood of the spherical regions, the same ellipse argument as in the proof of Proposition~\ref{polygon} applies to show that meridians are half-geodesics.

We now show that only finitely many of the $M_{2n,i}$ can admit a non-meridian half-geodesic. By contradiction assume that there exists a subsequence (again called $M_{2n,i}$) such that each $M_{2n,i}$ admits a non-meridian half-geodesic $\gamma_i \colon S^1 \to M_{2n,i}$. Then by Theorem~\ref{conv2} a subsequence of the $\gamma_i$ converge point-wise to a meridian half-geodesic $\gamma \colon S^1 \to X_{2n}$ (note by Proposition~\ref{no_short_closed} that $\gamma$ can not be the trivial curve). We use the fact that the lengths of the $\gamma_i$ are uniformly bounded to conclude that the $\gamma_i$ converge uniformly to $\gamma$. Thus for $i$ sufficiently large the $\gamma_i$ are the concatenation of straight-line segments on each face passing perpendicularly through parallel cylindrical edges. As in the proof of Proposition~\ref{polygon} such a non-meridian half-geodesic $\gamma_i$ contains a pair of off-center points between which it fails to minimize. We conclude that these $\gamma_i$ are not half-geodesics and therefore that only finitely many of the $M_{2n,i}$ can admit more than $n$ half-geodesics.
\end{proof}

\begin{prop}\label{ellipsoid2} There exists a sequence $M_i$ of triaxial ellipsoids, each admitting exactly one half-geodesic, which converge smoothly to the standard Riemannian sphere, all of whose geodesics are half-geodesics.
\end{prop}
\begin{proof} 

We apply the following result by Morse \cite[Theorem 109]{Ber}: given any L (think of L as very large) there is an $\epsilon > 0$ such that any triaxial ellipsoid with axis lengths satisfying $1<a<b<c< 1+\epsilon$ has all of its periodic geodesics of length larger than L except for the three sections by the coordinate planes. Choose a sequence $\epsilon_i \to 0$ and let $M_i$ be a triaxial ellipsoid with axis lengths satisfying $1<a_i<b_i<c_i< 1+\epsilon_i$. We have that the only short closed geodesics on $M_i$ are the three sections by the coordinate planes. Two of these sections will contain the largest axis and will fail to minimize between the antipodal points on their second axis; the other coordinate section provides a shorter path between these antipodes. The remaining coordinate section is the only half-geodesic. 
\end{proof}

\begin{proof}[Proof of Theorem~\ref{exact}] For $n=0$ we have the result of Theorem~\ref{none}. For $n=1$ we have the result of Proposition~\ref{ellipsoid2}. For $n \geq 2$ we have the result of Proposition~\ref{even}.
\end{proof}

%
%
\section{Acknowledgements} 
The author would like to thank Carolyn Gordon and Craig Sutton for their guidance through all stages of the research process. The author would also like to thank Christina Sormani for suggesting the original problem and for helpful discussions and direction. Finally the author would like to thank the referee for helpful suggestions in shortening the paper.

\nocite{*}
\bibliographystyle{plain}
\bibliography{polygon}

\end{document}